\documentclass{hha}

\begin{document}

\title{An Upper Bound for the Depth of Some Classes of Polyhedra}


\author{Mojtaba~Mohareri}

\email{m.mohareri@stu.um.ac.ir}

%
%
\address{Department of Pure Mathematics, Center of Excellence in Analysis on Algebraic Structures,
         Ferdowsi University of Mashhad, P.O.Box 1159-91775, Mashhad, Iran.}


\author{Behrooz~Mashayekhy}
\email{bmashf@um.ac.ir}
\address{Department of Pure Mathematics, Center of Excellence in Analysis on Algebraic Structures,
         Ferdowsi University of Mashhad, P.O.Box 1159-91775, Mashhad, Iran.}
         \def\Speaker{$^{*}$\protect\footnotetext{$^{*}$ C\lowercase{orresponding author.}}}
\author{Hanieh~Mirebrahimi\Speaker}

\email{h$_{-}$mirebrahimi@um.ac.ir}
\address{Department of Pure Mathematics, Center of Excellence in Analysis on Algebraic Structures,
         Ferdowsi University of Mashhad, P.O.Box 1159-91775, Mashhad, Iran.}





\begin{abstract}
 K. Borsuk in the seventies introduced  the notions of  capacity and  depth of compacta together with some relevant problems. In this paper, first, we introduce the concepts of the (strong) capacity and the (strong) depth  of an object in an arbitrary category. Then in the category of groups, we  compute the (strong) capacity and the (strong) depth  of  some well-known groups.  Finally, we find an upper bound for the depth of some classes of finite polyhedra which  generalizes a result of D. Kolodziejczyk in this subject.
\end{abstract}

\received{Month Day, Year}   
\revised{Month Day, Year}    
\published{Month Day, Year}  
\submitted{}      
\volumeyear{} 
\volumenumber{} 
\issuenumber{}   
\startpage{1}     
\articlenumber{1} 
\owner{}

\maketitle


\section{Introduction and Motivation}
Throughout this paper,  each polyhedron and CW-complex is assumed to be finite and connected. Also, by a map between two CW-complexes we mean here a cellular one.

 K. Borsuk in \cite{So} introduced  concepts of the capacity and the depth of a compactum together with some relevant problems. Indeed, the capacity $C(A)$ of a compactum $A$ is the cardinality of the class of the shapes of all compacta $X$ such that $Sh(X)\leqslant Sh(A)$, and the depth $D(A)$ of a compactum $A$ is the least upper bound of the lengths of all chains  $\mathcal{S}h(X_{k}) <\cdots < \mathcal{S}h(X_{2})  < \mathcal{S}h(X_{1}) \leqslant \mathcal{S}h(A)$, where $\mathcal{S}h(X)<\mathcal{S}h(Y)$ denotes that $\mathcal{S}h(X)\leqslant \mathcal{S}h(Y)$ holds but $\mathcal{S}h(Y)\leqslant \mathcal{S}h(X)$ fails. If this upper bound is infinite, we write $D(A)=\mathcal{N}_0$ (for more details, see \cite{Mar}). It is clear that $D(A) \leq C(A)$ for each compactum $A$.

For polyhedra, the notions shape and shape domination in the above definition can be replaced by the notions homotopy type and homotopy domination, respectively. Consequently, the capacity and depth in shape category of compacta are the same as in the homotopy category of CW-complexes and homotopy classes of cellular maps between them (see \cite{K4}).  S. Mather in \cite{17} proved that every polyhedron dominates only countably many different homotopy types (see also \cite{Hol}).

 Note that there exist polyhedra with the same capacities but different depths. For instance, one can consider $P_1 =\mathbb{T}\# \mathbb{T}$ and $P_2 =\mathbb{S}^1 \times \mathbb{S}^2$, where $\#$ denotes the connected sum operation. One can observe that $C(P_1 )=C(P_2 )=4$ but $4=D(P_1 )\neq D(P_2 )=3$ (see \cite{Mah}).

The answers to most of problems of \cite{So} are given by D. Kolodziejczyk in \cite{K2}.  In \cite{K7}, answering the question (4) of \cite{So}: ``Is it true that the capacity of every finite polyhedron is finite? '', Kolodziejczyk showed that there is a polyhedron dominating infinitely many different homotopy types of polyhedra. Kolodziejczyk in \cite{K3}, gave an affirmative answer to question (8) of \cite{So}: ``Is there a compactum with infinite capacity and finite depth?''. Indeed, by a result of Kolodziejczyk,  for every non-abelian poly-$\mathbb{Z}$-group $G$ and an integer $n\geq 3$, there exists a polyhedron $P$ such that $\pi_1 (P)\cong G$, $\dim P=n$, $C(P)$ is infinite, but $D(P)$ is finite. Accordingly, the examples of polyhedra with infinite capacity are even frequent. In \cite{K6}, Kolodziejczyk also  proved that every polyhedron with virtually polycyclic fundamental group has finite depth.

 Moreover, it is investigated the finiteness conditions for capacities of some polyhedra by Kolodziejczyk (see \cite{K5, K4}). For instance, a polyhedron $Q$ with finite fundamental group $\pi_1 (Q)$, and a polyhedron $P$ with abelian fundamental group $\pi_1 (P)$ and finitely generated homology groups $H_i (\tilde{P})$, for $i\geq 2$,  have finite capacities (hence finite depths), where $\tilde{P}$ denotes the universal covering space of $P$.

 Note that the capacity of a polyhedron is a homotopy invariant, i.e., for polyhedra $X$ and $Y$ with the same homotopy type, $C(X)=C(Y)$.  Hence it seems interesting to find polyhedra with finite capacities and compute the capacities of some of their well-known spaces.

Note that every compactum shape dominated by some polyhedron $P$ has the shape of some $P$-like compactum (for the definition of $P$-like, see \cite{Ma}). This is a corollary of some more general results of Holsztynski, Mardesic and Segal (see, for example, \cite[Ch.II.9, Lemma 4 on p. 222]{Mar} or \cite[Lemma 7]{KOW}). Therefore, the capacities of some polyhedra $P$ may be easily obtained from the classifications of P-like compacta. Such classifications are known for the $n$-dimensional torus, $\mathbb{S}^n$, $\mathbb{CP}^n$, $\mathbb{RP}^n$ (\cite{EGM}, \cite{Ke}, \cite{MS1}, \cite{W}, \cite{HS1}) and for a few other spaces. In the paper \cite{W} one can find shape classifications of $P$-like compacta for finite wedges of spheres $\mathbb{S}^n$ ($n \geq 2$). Two compacta of this kind, $X=\bigvee_i \mathbb{S}^n$ and $Y=\bigvee_j \mathbb{S}^n$, have the same shape if and only if $\check{H}^n (X; \mathbb{Z}) = \check{H}^n (Y; \mathbb{Z}$) (Cech cohomology groups) \cite[Theorem 2]{W}. As a consequence, the capacities of finite wedges of spheres $\mathbb{S}^n$ ($n \geq 2$) can be easily computed. It is also easily seen that the capacity of the wedge of $k$ copies of $\mathbb{S}^1$ is $k+1$.

M. Abbasi and B. Mashayekhy in \cite{Mah} computed the capacities of compact 2-dimensional manifolds. They showed that the capacities (and also the depths) of a compact orientable surface of genus $g\geq 0$ and a compact non-orientable surface of genus $g>0$ are equal to $g+2$ and $[\frac{g}{2}]+2$, respectively. The authors, in \cite{Moh1} computed the capacities of  finite wedges of spheres of various dimensions. Indeed, we showed that the capacity (and also the depth) of $\bigvee_{n\in I} (\vee_{i_n} \mathbb{S}^n)$ is equal to $\prod_{n\in I}(i_n +1)$, where $\vee_{i_n} \mathbb{S}^n$ denotes the wedge of $i_n$ copies of $\mathbb{S}^n$, $I$ is a finite subset of $\mathbb{N}$ and $i_n \in \mathbb{N}$. Also, in \cite{Moh},  we   computed the capacity of the product of two spheres of the same or different dimensions and the capacities of lens spaces which are a class of closed orientable 3-manifolds.

 In Section 2, we define  concepts of the (strong) capacity and the (strong) depth for an $A\in Obj\; (\mathcal{C})$, where  $\mathcal{C}$ denotes an arbitrary category. Also we prove some facts concerning relationships between concepts of the capacity and the depth in both ordinary and strong cases. In Section 3, in the category of groups, we compute the (strong) capacities and (strong) depths of two well-known groups: the free groups of finite rank and the finitely generated abelian groups. Also, we show that every virtually polycyclic group has finite depth.   Finally, in Section 4,  we generalize a theorem of  Kolodziejczyk (see \cite[Theorem 2]{K6}) which  states that every polyhedron with virtually polycyclic fundamental group has finite depth. Also, we present an upper bound for the depth of some classes of polyhedra.
\section{Strong capacity and strong depth in a category}
In what follows, $\mathcal{C}$ denotes an arbitrary category.

Recall that a domination in a given category $\mathcal{C}$ is a morphism $g : Y \longrightarrow X$,
$X, Y \in Obj \mathcal{C}$, for which there exists a morphism $f : X \longrightarrow Y$ of $\mathcal{C}$ such that $g\circ f = id_X$. In this case, we say that $X$ is dominated by $Y$ and we write $X \leqslant^d Y$ (rather than the notation ``$\leqslant$'', see for example \cite{K6}). Also,  $X$ is called  strongly dominated by $Y$ and is denoted by  $X <^s Y$  if $X\leqslant^d Y$ holds but $Y\leqslant^d X$ fails. Moreover, we say that $X$ is properly dominated by $Y$ and we denote it by $X<^p Y$ if $X\leqslant^d Y$ but $X \cong Y$ fails.

Clearly,  if $X,A \in Obj\; \mathcal{C}$ and $X<^s A$, then $X<^p A$ but the converse does not hold in general. For example, in the category of groups, there exists a countable torsion-free abelian group $A$ such that $A\cong A\oplus A\oplus A$ but $A\not \cong A\oplus A$ (see \cite{Fuchs}). It follows that $A\leqslant^d A\oplus A$ and  $A\oplus A \leqslant^d  A$ but $A\not \cong A\oplus A$. Therefore, $A<^p A\oplus A$ but $A\not <^s A\oplus A$.   Also, in the homotopy category of CW-complexes and homotopy classes of cellular maps between them, there exist two CW-complexes $X$ and $Y$ for which $X\leqslant^d Y$ and $Y\leqslant^d  X$ but $X\not \simeq Y$ (see \cite{Ste}). Accordingly, $X<^p Y$ but $X\not <^s Y$.
\begin{definition}
The capacity $C(A)$ of an $A \in Obj\; \mathcal{C}$ is the cardinality of the class of isomorphism classes of all the $X \in Obj\; \mathcal{C}$ such that $X\leqslant^d A$. Also, the strong capacity $SC(A)$ of an $A \in Obj\; \mathcal{C}$ is  the cardinality of the class of isomorphism classes of all  $X \in Obj\; \mathcal{C}$ such that $X<^s A$. If there is no such $X$, we write $SC(A)=0$.
\end{definition}
The previous definition of the capacity of a compatum in the shape category of compacta coincides with Borsuk's definiton of the capacity (see \cite{So}).
 \begin{definition}
\begin{enumerate}
 \item
A chain $X_k <^p \cdots <^p X_1 \leqslant^d A$, where $X_i \in Obj\; \mathcal{C}$ for $i = 1,\cdots , k$, is called a chain of length $k$ for $A \in Obj\; \mathcal{C}$.
 \item
The  depth $D(A)$ of an $A\in Obj\; \mathcal{C}$ is the least upper bound of the lengths of all chains  for $A$. If this upper bound is infinite, we write $D(A)=\mathcal{N}_0$.
 \item
A chain  $X_k <^s \cdots <^s X_1 \leqslant^d A$, where $X_i \in Obj\; \mathcal{C}$ for $i = 1,\cdots , k$, is called an $s$-chain of length $k$ for $A \in Obj\; \mathcal{C}$.
 \item
  The strong depth $SD(A)$ of $A$ is the least upper
bound of the lengths of all $s$-chains for $A$. If this upper bound is infinite, we write $SD(A) =\mathcal{N}_0$.
\end{enumerate}
 \end{definition}
 Note that our definition of strong depth of a compactum in the shape category of compacta coincides with Borsuk's definition of the depth of a compactum (for more details, see \cite{So}).

It is easy to see that $A\in Obj\; \mathcal{C}$ has finite (strong) depth $k$ if and only if there exists (an) a $(s-)$chain of length $k$ for $A$ but  there is no $(s-)$chain of length $k+1$, equivalently, there exists (an) a $(s-)$chain of length $k$ for $A$ and for every $(s-)$chain
\[
X_k <^p \cdots <^p X_1 \leqslant^d X_0 =A \quad  (X_k <^s \cdots <^s X_1 \leqslant^d X_0 =A)
\]
 of  length of $k$ and for every  $X_{k+1}\leqslant^d X_k$, we have $X_{k+1}\cong X_k$ ($X_{k}\leqslant^d X_{k+1}$).

It is clear that $SC(A)\leq C(A)-1$, for every  $A\in Obj\; \mathcal{C}$.
\begin{definition}
Two objects $X$ and $Y$ in  $\mathcal{C}$ are called $d$-equal and denoted by $X=_d Y$  if $X\leqslant^d Y$ and $Y\leqslant^d X$.
\end{definition}
\begin{definition}
We say that  $A\in Obj\; \mathcal{C}$  has domination equality ($d$-equlaity for short) property if  for every  $X\in Obj\; \mathcal{C}$ with $X=_d A$, we have $X\cong A$.
\end{definition}
By \cite[Theorem 3]{K1}, if $X\in ANR$ and $\pi_1 (X)$ is virtually polycyclic fundamental group, then $X$ has $d$-equality property (also see \cite{KOMEN}). Also,  every Hopfian group has  $d$-equality property (see Prposition \ref{Hop}).
\begin{lemma}\label{1}
Suppose that $A\in Obj \; \mathcal{C}$  has $d$-equality property. If $X\in Obj\; \mathcal{C}$ is properly  dominated by $A$, then $X$ is strongly dominated by $A$.
\end{lemma}
\begin{proof}
By contrary, assume that $X$ is not strongly  dominated by $A$, i.e., $X\leqslant^d A$ and $A\leqslant^d X$. Since $A$ has  $d$-equality property, we have $A\cong X$. But this is  a contradiction to $X<^p A$.
\end{proof}
\begin{proposition}\label{SC}
If  $A\in Obj \; \mathcal{C}$ has $d$-equality property, then $SC(A)=C(A)-1$.
\end{proposition}
\begin{proof}
It follows from Lemma \ref{1}.
\end{proof}
\begin{lemma}\label{trans}
If $Y\leqslant^d X$ and $X<^s A$ for $A,X,Y\in Obj \; \mathcal{C}$, then $Y<^s A$.
\end{lemma}
\begin{proof}
Since $Y\leqslant^d X$ and $X \leqslant^d A$, then $Y\leqslant^d A$. By contrary, assume that $A\leqslant^d Y$. Then $A\leqslant^d X$ which is a contradiction to $X<^s A$.  Thus we have $Y<^s A$.
\end{proof}
\begin{remark}
From the previous lemma, one can conclude that the relation ``$<^s$'' is a transitive relation. Clearly,  the relation ``$\leqslant^d$'' is also transitive. But the relation ``$<^p$'' is not necessarily transitive. Indeed, as mentioned before, there exist two CW-complexes $X$ and $Y$ for which $X\leqslant^d Y$ and $Y\leqslant^d  X$ but $X\not \simeq Y$. It follows that   $X<^p Y$ and $Y <^p X$ but $X\simeq X$.
\end{remark}
\begin{lemma}\label{joda}
Suppose that $A\in Obj \; \mathcal{C}$. If $X_k <^s \cdots <^s X_1 \leqslant^d X_0 =A$ is an  $s$-chain of length $k$ for $A$, then $X_i \not \cong X_j$ for all $1\leq i\neq j\leq k$.
\end{lemma}
\begin{proof}
By Lemma \ref{trans}, $X_k <^s X_i$ for all $1\leq i\leq k-1$. Then by definition we have $X_k \not \cong X_i$ for all $1\leq i\leq k-1$. Similarly, one can show that for each $k-1 \leq i \leq 2$, we have $X_i \not \cong X_j$ for all $i+1 \leq j \leq 1$. Thus the proof is finished.
\end{proof}
\begin{proposition}\label{sdc}
For each $A\in Obj\; \mathcal{C}$, $SD(A)\leq \min \{D(A),C(A)\}$.
\end{proposition}
\begin{proof}
By definition, it is clear that $SD(A)\leq D(A)$. Suppose that $C(A)=k<\infty$. By contrary, assume that there is the following  $s$-chain for $A$
\[
X_{k+1}<^s X_k <^s \cdots <^s X_1 \leqslant^d X_0 =A .
\]
 Since the relation ``$\leqslant^d$'' is transitive, so $X_i \leqslant^d A$ for all $i=1,\cdots ,k+1$. By Lemma \ref{joda}, $X_i$'s are mutually non-isomorphic for all $i=1,\cdots ,k+1$. It follows that $C(A)\geq k+1$ which is a contradiction to the hypothesis. Then there is no $s$-chain of  length $k+1$ for $A$. This proves that $SD(A)\leq k$.
\end{proof}
\begin{proposition}\label{2.11}
Suppose that $A\in Obj\; \mathcal{C}$. If $D(A)=k$, then each sequence $\cdots \leqslant^d X_i \leqslant^d \cdots  \leqslant^d X_1 \leqslant^d A$ contains at most $k$ different objects (up to isomorphism).
\end{proposition}
\begin{proof}
By contrary, assume that there exists a sequence $\cdots \leqslant^d X_i \leqslant^d \cdots \leqslant^d X_1 \leqslant^d A$ contains at least $k+1$ different objects (up to isomorphism), called $X_{i_{k+1}},X_{i_k},\cdots ,X_{i_1}$ with $i_{k+1}<i_{k}<\cdots <i_1$. It follows that the sequence $X_{i_{k+1}}<^p X_{i_k}<^p \cdots <^p X_{i_1}\leqslant^d A$  is a chain of length $k+1$ for $A$ which is a contradiction to $D(A)=k$.
\end{proof}
\begin{proposition}\label{7575}
Suppose that $A\in Obj \; \mathcal{C}$. If there exists an integer $k$ such that each sequence $\cdots \leqslant^d X_i \leqslant^d \cdots  \leqslant^d X_1 \leqslant^d A$ contains at most $k$ different objects (up to isomorphism), then $SD(A)\leq k$.
\end{proposition}
\begin{proof}
By contrary, assume that there exists the following  $s$-chain  for $A$
\[
X_{k+1} <^s \cdots <^s X_1 \leqslant^d A.
\]
By Lemma \ref{joda}, we have $X_i \not \cong X_j$ for all $1\leq i\neq j\leq k+1$.  Then  the sequence
\[
\cdots \leqslant^d X_i \leqslant^d \cdots \leqslant^d X_1 \leqslant^d A
\]
with $X_i =X_{k+1}$ for all $i\geq k+1$, contains at least $k+1$ different objects (up to isomorphism) which is a contradiction to the hypothesis. Therefore, $SD(A)\leq k$.
\end{proof}
\begin{lemma}\label{deq}
Suppose that $A\in Obj\; \mathcal{C}$. If every $X\leqslant^d A$ has $d$-equality property, then every chain for $A$ is an $s$-chain.
\end{lemma}
\begin{proof}
 Asume that $X_k <^p \cdots <^p X_1 \leqslant^d A$ is a chain of length $k$ for $A$. Assume that $X_{i-1} \leq X_{i}$ for some $2\leq i\leq k$. Since $X_i$ has $d$-equality property by the hypothesis, then $X_i \cong X_{i-1}$ which is a contradiction to $X_i <^p X_{i-1}$. So $X_{i}<^s X_{i-1}$ for every $2\leq i\leq k+1$. Hence
\[
X_{k+1}<^s X_k <^s \cdots <^s X_1 \leqslant^d A
\]
is an $s$-chain of  length $k$ for $A$.
\end{proof}
Note that by \cite[Theorem 3]{K1}, if $A\in ANR$ and $\pi_1 (A)$ is virtually polycyclic, then every $X\leqslant^d A$ has $d$-equality property.
\begin{proposition}\label{e}
Suppose that $A\in Obj\; \mathcal{C}$. If every $X\leqslant^d A$ has $d$-equality property, then $D(A)=SD(A)$.
\end{proposition}
\begin{proof}
It follows from Lemma \ref{deq}.
\end{proof}
\begin{lemma}\label{2}
Suppose that $A\in Obj\; \mathcal{C}$. If $D(A)<\infty$, then every $X\leqslant^d A$ has  $d$-equality property.
\end{lemma}
\begin{proof}
Suppose that $D(A)=k$. First, we show that $A$ itself has $d$-equality property.  By contrary, assume that there exists $X\in Ob\; \mathcal{C}$  for which $X\leqslant^d A$ and $A\leqslant^d X$ but $A\not \cong X$. Then the sequence
\[
\underbrace{A<^p X<^p A<^p \cdots <^p X<^p A}_{\text{of length}\; k+1},
\]
is a chain of length $k+1$ for $A$ which is a contradiction to $D(A)=k$.

Now, we know that
\[
``Y\leqslant^d X \quad \text{implies}\quad D(Y)\leq D(X)\text{''}
\]
and so the result holds.
\end{proof}
\begin{corollary}
Suppose that $A\in Obj\; \mathcal{C}$. If $D(A)<\infty$, then $D(A)=SD(A)$.
\end{corollary}
\begin{proof}
It is concluded from Lemma \ref{2} and Proposition \ref{e}.
\end{proof}
\section{The (strong) capacities and (strong) depths of some classes of groups}
Recall that a homomorphism $g :G\longrightarrow H$ of groups is an $r$-homomorphism if there
exists a (converse) homomorphism $f :H\longrightarrow G$ such that $g\circ f = id_H$. Then $H$ is called  an $r$-image of $G$. Note that in the categroy of groups, concepts of the domination and the $r$-homomorphism are the same. So the group $H$ is an $r$-image of the group $G$ if and only if $H\leqslant^d G$. Also, if $H$ is (strognly) properly dominated by $G$, then we say that $H$ is an ($s$-image) proper $r$-image of $G$.

Recall that an endomorphism $h:G\longrightarrow G$ of a group $G$ is called an idempotent if $h\circ h=h$.  It is clear that if $H$ is an $r$-image of $G$ with $r$-homomorphism $g :G\longrightarrow H$ and converse homomorphism $f :H \longrightarrow G$, then $f\circ g:G\longrightarrow G$ is an idempotent. By \cite[Corollary]{Hol}, $C(G)\leq e(G)$, where $e(G)$ denotes the number of idempotent endomorphisms on $G$.
\begin{proposition}
Let $G$ be a finite group of order $n$. Then $C(G)\leq n^{(\frac{1}{2}+o(1))\log_2 (n)}$.
\end{proposition}
\begin{proof}
Clearly, $e(G)$  is equal to the number of pairs of subgroups $H,K$ such that $G$  is the semidirect product of $H$ and $K$. By \cite[Corollary 1.6]{Bor}, the number of subgroups of $G$ is bounded by $n^{(\frac{1}{4}+o(1))\log_2 (n)}$. Squaring gives an upper bound on the total number of pairs $(H,K)$ and hence on $e(G)$ as follows:
\[
e(G)\leq n^{(\frac{1}{2}+o(1))\log_2 (n)}.
\]
By the fact that $C(G)\leq e(G)$, the proof is finished.
\end{proof}

Recall that a group $G$ is Hopfian if every epimorphism $f :G\longrightarrow G$ is an automorphism
(equivalently, $N=1$ is the only normal subgroup for which $G/N\cong G$).
\begin{proposition}\label{Hop}
 Every Hopfian group has  $d$-equality property.
\end{proposition}
\begin{proof}
Assume that $G$ is a Hopfian group. Also, assume that for a group $H$, we have $G\leqslant^r H$ and $H\leqslant^r G$. Then there are homomorphisms $f_1 :H\longrightarrow G$ and $g_1 :G\longrightarrow H$ with $g_1 \circ f_1 =id_H$ and homomorphisms $f_2 :G\longrightarrow H$ and $g_2 :H\longrightarrow G$ with $g_2 \circ f_2 =id_G$.  It follows that $g_2 \circ g_1 :G\longrightarrow G$ is an epimorphism and so an isomorphism ($G$ is Hopfian). Now one can conclude that $g_1$ is a monomorphism and so an isomorphism, hence $G\cong H$.
\end{proof}
\begin{corollary}\label{Hopf}
For a Hopfian group $G$, $SC(G)=C(G)-1$.
\end{corollary}
\begin{proof}
It follows from Propositions \ref{SC} and \ref{Hop}.
\end{proof}
Now, we compute the (strong) capacities and (strong) depths of free groups of finite ranks and finitely generated abelian groups.
\begin{lemma}\label{11}
Let $F$ be a free group of finite rank. If $H$ is an $r$-image (a proper $r$-image) of $F$, then $H$ is a free group with $rank (H)\leq rank(F)$ $(rank (H)< rank (F))$.
\end{lemma}
\begin{proof}
Put $rank(F)=k$. Assume that $H$ is an $r$-image of $F$ with  $r$-homomorphism $g:F\longrightarrow H$ and  converse homomorphism $f:H\longrightarrow F$. Define homomorphisms $\bar{g}:\frac{F}{F'}\longrightarrow \frac{H}{H'}$ and $\bar{f}:\frac{H}{H'}\longrightarrow \frac{F}{F'}$ by $\bar{g}(xF')=g(x)H'$ and $\bar{f}(xH')=f(x)F'$, respectively. Since $g\circ f=id_H$, then $\bar{g}\circ \bar{f}=id_{\frac{H}{H'}}$. It follows that $\frac{H}{H'}$ is isomorphic to a subgroup of the free abelian group $\frac{F}{F'}$. But we know that the ranks of free groups $F$ and $H$ are equal to the ranks of free abelian groups $\frac{F}{F'}$ and $\frac{H}{H'}$, respectively. On the other hand, $rank (\frac{H}{H'})$ is less than or equal to $rank(\frac{F}{F'})=k$. Therefore, $rank\; (H)\leq k$.

Now if $H$ is a proper $r$-image of $F$, then  $H\not \cong F$ by definition. So it follows that $rank(H)<k$.
\end{proof}
\begin{proposition}\label{cf}
Let $F$ be a free group of the rank $k$. Then the following statements satisfy:
\begin{enumerate}
\item
The strong capacity and the capacity of $F$ are equal to $k$ and $k+1$, respectively;
\item
The strong depth and the depth of $F$ are equal to $k+1$.
\end{enumerate}
\end{proposition}
\begin{proof}
(1)  By Lemma \ref{11}, each $r$-image of $F$ is a free group of rank less than or equal to $k$.  Also, it is easy to show that each free group of rank $i$ is an $r$-image of $F$ for all $0\leq i\leq k$. Therefore, we conclude that $C(F)=k+1$. But $F$ is a Hopfian group which implies that $SC(F)=k$, by Corollary \ref{Hopf}.

(2) First, we show that $SD(F)=k+1$. By part $(i)$ and Proposition \ref{sdc}, we have $SD(F)\leq C(F)=k+1$.  Assume that $\{ x_1 ,\cdots ,x_k \}$ is a generator set of $F=F_0$. Then it is not hard to show that  the free group $F_i$ generated by $\{ x_1 ,\cdots ,x_{k-i+1}\}$ is an $r$-image of the free group $F_{i-1}$ generated by $\{ x_1 ,\cdots ,x_{k-i+2}\}$, for all $2\leq i\leq k+1$. Note that if $F_{i_0}\leqslant^d F_{i_0 +1}$ for some $i_0 \in \{ 1,\cdots ,k-1\}$, then by Proposition \ref{Hop}, we have $F_{i_0}\cong F_{i_0 +1}$ which is a contradiction.  Hence, one can obtain the following  $s$-chain of  length $k+1$ for $F$
\[
 1=F_{k+1}<^s F_k <^s \cdots <^s F_2 <^s F_1 \leqslant^d F_0 = F.
\]
Thus one can conclude that $SD(F)=k+1$.

Now we show that $D(F)=k+1$. Assume that there exists the following chain of length $k+2$ for $F$
\[
H_{k+2} <^p H_{k+1} <^p \cdots <^p H_2 <^p H_1 \leqslant^d H_0 =F.
\]
By Proposition \ref{11},  $H_1$ is a free group with $rank(H_1 )\leq rank(F)=k$. Again by Proposition \ref{11}, $H_2$ is a free group  with $rank(H_2 )< rank(H_1 )$. By a similar argument, $H_i$ is a free group with $rank(H_i )< rank(H_{i-1})$ for all $2\leq i\leq k+2$. Then we have
\[
 r_{k+2} < \cdots < r_2 < r_1 \leq k,
\]
where $r_i =rank(H_i )$ for all $1\leq i\leq k+2$.

It is clear that $rank(H_{k+1})=0$. Thus we have $H_{k+2} =H_{k+1}=1$ which is a contradiction to $H_{k+2}<^p H_{k+1}$. It follows that there is no chain of  length $k+2$, and so $D(F)\leq k+1$. On the other hand, $k+1=SD(F)\leq D(F)$ which implies that $D(F)=k+1$. Thus the proof is complete.
\end{proof}
\begin{lemma}\label{p1}
Let $G$ be a finitely generated abelian group of the  form
\[
\mathbb{Z}_{p_{1}^{\alpha_1}}^{(k_{1})}\oplus \mathbb{Z}_{p_{2}^{\alpha_2}}^{(k_{2})} \oplus \cdots \oplus \mathbb{Z}_{p_{n}^{\alpha_n}}^{(k_{n})},
\]
where for $i\neq j$, $p_{i}^{\alpha_i}\neq p_{j}^{\alpha_j}$, $p_i$'s are prime numbers, $\alpha_i$'s are non-negative integers, $\mathbb{Z}_{p_{i}^{\alpha_i}}^{(k_{i})}$ is the direct sum of $k_i$ copies of $\mathbb{Z}_{p_{i}^{\alpha_i}}$, and $\mathbb{Z}_{1}=\mathbb{Z}$. Then the number of direct summands of $G$, up to isomorphism,  is equal to
\[
(k_{1}+1)\times \cdots \times (k_{n}+1).
\]
\end{lemma}
\begin{proof}
The proof is in three steps.

Step One. For each $1\leq i\leq n$, the number of direct summands of $\mathbb{Z}_{p_{i}^{\alpha_i}}^{(k_{i})}$, up to isomorphism, is equal to $k_i +1$.

For this, it is obviuos that for every $0\leq t \leq k_i$, $\mathbb{Z}_{p_{i}^{\alpha_i}}^{(t)}$ is a direct summand of $\mathbb{Z}_{p_{i}^{\alpha_i}}^{(k_{i})}$ and for each $0 \leq t \neq t' \leq k_i$, we have $\mathbb{Z}_{p_{i}^{\alpha_i}}^{(t)} \not \cong \mathbb{Z}_{p_{i}^{\alpha_i}}^{(t')}$. Now, suppose that $C$ is a direct summand of $\mathbb{Z}_{p_{i}^{\alpha_i}}^{(k_{i})}$. There exists a subgroup $D$ of $\mathbb{Z}_{p_{i}^{\alpha_i}}^{(k_{i})}$ such that $\mathbb{Z}_{p_{i}^{\alpha_i}}^{(k_{i})} \cong C \oplus D$. By \cite[Corollary 2.1.7]{Hu},  $C$ is a finitely generated abelian group. Suppose that $C\cong \mathbb{Z}_{q_{1}^{\beta_1}}^{(l_{1})}\oplus \cdots \oplus \mathbb{Z}_{q_{s}^{\beta_s}}^{(l_{s})}$. Since $C$ is a direct summand of $\mathbb{Z}_{p_{i}^{\alpha_i}}^{(k_{i})}$ and for every $1\leq j \leq s$, $\mathbb{Z}_{q_{j}^{\beta_j}}^{(l_{j})}$ is a direct summand of $C$, so for every $1\leq j\leq s$, $\mathbb{Z}_{q_{j}^{\beta_j}}^{(l_{j})}$ is a direct summand of $\mathbb{Z}_{p_{i}^{\alpha_i}}^{(k_{i})}$. Now, by uniqueness of decomposition of finitely generated abelian groups \cite[Theorem 2.2.6, (iii)]{Hu}, for any $j=1,\cdots ,s$, there exists an $i=1,\cdots ,n$ such that  $q_j =p_i$ and $\beta_j =\alpha_i$. Hence, $C\cong \mathbb{Z}_{p_{i}^{\alpha_i}}^{(t)}$ for some $0\leq t\leq k_i$.

Step Two. The number of direct summands of $\mathbb{Z}_{p_{i}^{\alpha_i}}^{(k_{i})} \oplus \mathbb{Z}_{p_{j}^{\alpha_j}}^{(k_{j})}$ for $i\neq j$, up to isomorphism, is equal to $(k_i +1)(k_j +1)$.

It is easy to see that for every $0\leq t\leq k_i$ and $0\leq s\leq k_j$, $\mathbb{Z}_{p_{i}^{\alpha_i}}^{(t)}\oplus \mathbb{Z}_{p_{j}^{\alpha_j}}^{(s)}$ is a direct summand of $\mathbb{Z}_{p_{i}^{\alpha_i}}^{(k_{i})} \oplus \mathbb{Z}_{p_{j}^{\alpha_j}}^{(k_{j})}$. Now similar to Step One, suppose that $C$ is a direct summand of $\mathbb{Z}_{p_{i}^{\alpha_i}}^{(k_{i})} \oplus \mathbb{Z}_{p_{j}^{\alpha_j}}^{(k_{j})}$ and $D$ is a subgroup of $\mathbb{Z}_{p_{i}^{\alpha_i}}^{(k_{i})} \oplus \mathbb{Z}_{p_{j}^{\alpha_j}}^{(k_{j})}$ such that $\mathbb{Z}_{p_{i}^{\alpha_i}}^{(k_{i})} \oplus \mathbb{Z}_{p_{j}^{\alpha_j}}^{(k_{j})} \cong C\oplus D$. Suppose $C \cong\mathbb{Z}_{q_{1}^{\beta_1}}^{(l_{1})}\oplus \cdots \oplus \mathbb{Z}_{q_{s}^{\beta_s}}^{(l_{s})}$. Since for every $1\leq m\leq s$, $\mathbb{Z}_{q_{m}^{\beta_m}}^{(l_{m})}$ is a direct summand of $\mathbb{Z}_{p_{i}^{\alpha_i}}^{(k_{i})} \oplus \mathbb{Z}_{p_{j}^{\alpha_j}}^{(k_{j})}$, so similar to the above argument, $C\cong  \mathbb{Z}_{p_{i}^{\alpha_i}}^{(t)} \oplus  \mathbb{Z}_{p_{j}^{\alpha_j}}^{(s)}$ for some $0\leq t\leq k_i$ and $0\leq s \leq k_j$.

Step Three: the number of direct summands of $\mathbb{Z}_{p_{1}^{\alpha_1}}^{(k_{1})}\oplus \mathbb{Z}_{p_{2}^{\alpha_2}}^{(k_{2})} \oplus \cdots \oplus \mathbb{Z}_{p_{n}^{\alpha_n}}^{(k_{n})}$, up to isomorphism, is equal to $(k_1 +1)(k_2 +1) \cdots (k_n +1)$ which is obtained by induction from Step Two.
\end{proof}
\begin{proposition}\label{dfg}
Let $G$ be a finitely generated abelian group of the form
\[
\mathbb{Z}_{p_{1}^{\alpha_1}}^{(k_{1})}\oplus \mathbb{Z}_{p_{2}^{\alpha_2}}^{(k_{2})} \oplus \cdots \oplus \mathbb{Z}_{p_{n}^{\alpha_n}}^{(k_{n})},
\]
up to isomorphism, where for $i\neq j$, $p_{i}^{\alpha_i}\neq p_{j}^{\alpha_j}$, $p_i$'s are prime numbers, $\alpha_i$'s are non-negative integers, $\mathbb{Z}_{p_{i}^{\alpha_i}}^{(k_{i})}$ is the direct sum of $k_i$ copies of $\mathbb{Z}_{p_{i}^{\alpha_i}}$, and $\mathbb{Z}_{1}=\mathbb{Z}$. Then the following statements satisfy:
\begin{enumerate}
\item
The strong capacity and the capacity of $G$ are equal to $\big( (k_{1}+1)\times \cdots \times (k_{n}+1)\big)-1$ and $(k_{1}+1)\times \cdots \times (k_{n}+1)$, respectively.
\item
The strong depth and the depth of $G$ are equal to $\big( \sum_{i=1}^{m}k_i \big) +1$.
\end{enumerate}
\end{proposition}
\begin{proof}
(1) It is easy to show that for every arbitrary abelian group $A$,  there exists a one-to-one corresponding between the class of  isomorphism classes of all $r$-images of $A$ and the class of isomorphism classes of all direct summands of $A$. Accordingly, the capacity of $G$ is equal to  $(k_{1}+1)\times \cdots \times (k_{n}+1)$  by Lemma \ref{p1}. Now, by the fact that every finitely generated abelian group is Hopfian and by Corollary \ref{Hopf}, we have $SC(G)=\big( (k_{1}+1)\times \cdots \times (k_{n}+1)\big)-1$.

(2) First, we prove that $D(G)=\big(\sum_{i=1}^{n}k_i \big) +1$. Suppose that $H_1 $ is an arbitrary $r$-image of $H_0 =G$ with $r$-homomorphism $r_0 :H_0 \longrightarrow H_1$. Since $H_1$ is an $r$-image of abelian group $G$, then it is isomorphic to a direct summand of $G$. One can easily see that $H_1\cong \mathbb{Z}_{p_{i_1}^{\alpha_{i_1}}}^{(l_{i_1})}\oplus \cdots \mathbb{Z}_{p_{i_m}^{\alpha_{i_m}}}^{(l_{i_m})}$ for $\{ i_1 ,\cdots ,i_m \}\subseteq \{ 1,\cdots ,n\}$, where $l_{i_j}\leq k_{i_j}$ for all $j$.  Suppose that $H_2$ is a proper $r$-image of $H_1$. Since $H_2 \not \cong H_1$, one can conclude that $H_2 \cong \mathbb{Z}_{p_{j_1}^{\alpha_{j_1}}}^{(s_{j_1})}\oplus \cdots \mathbb{Z}_{p_{j_t}^{\alpha_{j_t}}}^{(s_{j_t})}$ for $\{ j_1 ,\cdots ,j_t \}\subsetneq \{ i_1 ,\cdots ,i_m \}$, where $s_{j_r}\leq l_{j_r}$ for all $r$.    Without loss of generality, we can suppose that
\[
H_2 \cong \mathbb{Z}_{p_{i_1}^{\alpha_{i_1}}}^{(l_{i_1})}\oplus \cdots \mathbb{Z}_{p_{i_m}^{\alpha_{i_m}}}^{(l_{i_m}-1)}.
\]
Now if
\[
H_{l}<^p \cdots <^p H_2 <^p H_1 \leqslant^d H_0 =G
\]
is an arbitrary $r$-chain for $G$ with $l=\sum_{j=1}^{m}l_{i_{j}}$, then with a similar argument and without loss of generality, we can assume that $H_{l}\cong \mathbb{Z}_{p_{1}^{\alpha_1}}$. But it is clear that every proper $r$-image  of $\mathbb{Z}_{p_{1}^{\alpha_1}}$ is trivial. This shows that there is no chain of  length $l+2$. Since $l\leq \sum_{i=1}^{n}k_i $, one can conclude that there is no chain of  length $\big( \sum_{i=1}^{n}k_i \big)+2$  and so $D(G)\leq \big( \sum_{i=1}^{n}k_i \big)+1$.

On the other hand, it is easy to see that the chain
\[
1<^p \mathbb{Z}_{p_{1}^{\alpha_{1}}}<^p \mathbb{Z}_{p_{1}^{\alpha_{1}}}^{(2)}<^p \cdots <^p  \mathbb{Z}_{p_{1}^{\alpha_{1}}}^{(k_1 )}<^p  \mathbb{Z}_{p_{1}^{\alpha_{1}}}^{(k_1 )}\oplus \mathbb{Z}_{p_{2}^{\alpha_{2}}}<^p \cdots  <^p \mathbb{Z}_{p_{1}^{\alpha_{1}}}^{(k_1 )}\oplus \cdots \oplus \mathbb{Z}_{p_{n}^{\alpha_{n}}}^{(k_n -1)}<^p G,\; (*)
\]
is an  $r$-chain of length $\big( \sum_{i=1}^{n}k_i \big) +1$. Therefore, $D(G)=\big(\sum_{i=1}^{n}k_i \big) +1$.

Now, we show that $SD(G)=\big( \sum_{i=1}^{n}k_i \big) +1$. Since $SD(G)\leq D(G)$, we have $SD(G)\leq \big( \sum_{i=1}^{n}k_i \big) +1$. On the other hand, since every finitely generated abelian group is Hopfian, by Proposition \ref{Hop}, the chain $(*)$ is an $s$-chain of  length $\big( \sum_{i=1}^{n}k_i \big) +1$ for $G$. Accordingly, $SD(G)=\big( \sum_{i=1}^{n}k_i \big) +1$.
\end{proof}
Note that one can easily see that the equality of $SD(G)=C(G)$ does not hold, in general (see, Proposition \ref{sdc}). It is enough to consider  $G=\mathbb{Z}_2 \oplus \mathbb{Z}\oplus \mathbb{Z}$ with $C(G)=6$, $C(G)=5$, $D(G)=SD(G)=4$.
\begin{theorem}\label{dvp}
The depth of every virtually polycyclic group is finite.
\end{theorem}
\begin{proof}
By \cite[Lemma 1]{K6},  there exists an integer $k$ such that any sequence $H_l \subseteq \cdots \subseteq H_1 \subseteq H_0 =G$ of subgroups of $G$ with $r$-homomorphisms $r_i : H_{i-1} \longrightarrow H_i$ for $i=1,\cdots , l$ contains at most $k$ distinct subgroups.

We claim that $D(G)\leq k$. By contrary, assume that there exists the following chain  of  length $k+1$  for $G$
\[
 H_{k+1} <^p H_k <^p \cdots  <^p H_1 \leqslant^d H_0 =G, \quad (*)
\]
 with  $r$-homomorphisms $g_i : H_{i-1} \longrightarrow H_i$ and converse homomorphisms $f_i :H_{i}\longrightarrow H_{i-1}$ for all $1\leq i\leq k+1$. Then we have the following sequence of subgroups for $G$
\[
f_1 f_2 \cdots f_{k+1}(H_{k+1}) \leqslant^d f_1 f_2 \cdots f_{k} (H_{k})\leqslant^d \cdots \leqslant^d f_2 f_1 (H_2 )\leqslant^d f_1 (H_1 )\leqslant^d H_0 =G, \quad (**)
\]
with retractions $r_i :f_1 f_2 \cdots f_{i-1} (H_{i-1})\longrightarrow f_1 f_2 \cdots f_{i} (H_{i})$ which are defined by
\[
r_i (f_1 (f_2 (\cdots (f_{i-1}(x)))))=f_1 (f_2 (\cdots (f_{i-1}(f_i (g_i (x)))))),
\]
 for all $1\leq i\leq k+1$. Clearly, the retraction $r_i$ is well-defined for all $1\leq i\leq k+1$.

Now we show that $f_1 f_2 \cdots f_{i}(H_{i})$'s are mutually distinct for all $1\leq i \leq k+1$.  Suppose that $f_1 f_2 \cdots f_{i} (H_{i})=f_1 f_2 \cdots f_{j} (H_j )$ for $1\leq j<i\leq k+1$. Then since $f_k$'s are monomorphisms, we have $f_{j+1}\cdots f_{i}(H_{i} )=H_j$. It follows  that the monomorphism $f_{j+1}\cdots f_{i}$ is onto and so is an isomorphism. Now since $H_j =Im(f_{j+1}\cdots f_{i})\subseteq Im(f_{j+1})$, $f_{j+1}$ is also an isomorphism, i.e. $H_{j+1}\cong H_j$ which is a contradiction to the chain $(*)$.  Consequently, we have  the sequence $(**)$ of subgroups of $G$  containing at least $k+1$ distinct subgroups which is a contradiction. Thus $D(G)\leq k$.
\end{proof}
\begin{remark}
Note that the converse of Theorem \ref{dvp} does not hold, in general. As an example, a free group of finite rank $k$ has the strong depth $k+1$ by part $(ii)$ of Proposition \ref{cf}. But we know that free groups of rank $\geq 2$ are not virtually polycyclic.
\end{remark}
\begin{proposition}\label{moji}
Let $G$ be a virtually polycyclic group.The depth of a finitely generated module $M$ over the integral group ring $\mathbb{Z}G$ is finite.
\end{proposition}
\begin{proof}
It is concluded from \cite[Lemma 2]{K6} and a similar argument used in the proof of Theorem \ref{dvp}.
\end{proof}
\section{A theorem concerning the depth of polyhedra}
In this section, we generalize a theorem of  Kolodziejczyk (see \cite[Theorem 2]{K6}) which  states that  every polyhedron with virtually polycyclic fundamental group has finite depth.   Note that for polyhedra, the notions shape and shape domination  can be replaced by the notions homotopy type and homotopy domination, respectively (see \cite{K4}).

\begin{lemma}\label{ASASI}
Let $H_{k}\leqslant^d \cdots \leqslant^d H_1 \leqslant^d H_0 =G$ be a chain of  length $k$ for a group $G$ with $r$-homomorphisms $g_i :H_{i-1}\longrightarrow H_i$, for all $1\leq i\leq k$. Also, suppose that $H_i$ is a Hopfian group $(0\leqslant i\leqslant k)$. If for some $1\leq j_0 \leq k$, $H_{j_0}\not \cong H_{j_0 -1}$, then $H_{j_0}\not \cong H_{i}$ for all $0\leq i<j_0$.
\end{lemma}
\begin{proof}
 By contrary, assume for some $j_0 \geq 2$, $H_{j_0}\not \cong H_{j_0 -1}$, but  $H_{j_0}\cong H_{i}$ for some $i<j_0 -1$. Then since $g_{j_0}\cdots g_{i+1}:H_i \longrightarrow H_{j_0}$ is an epimorhism between two isomorphic Hopfian groups, so one can conclude that $g_{j_0}\cdots g_{i+1}$ is an isomorphism. Since $Ker(g_{i+1})\subseteq Ker (g_{j_0}\cdots g_{i+1})$, $g_{i+1}$ is also isomorphism. It follows that $g_{j_0}\cdots g_{i+2}$ is an isomorphism. With a similar argument, we can conclude  that $g_{j_0}:H_{j_0 -1}\longrightarrow H_{j_0}$ is an isomorphism which is a contradiction to $H_{j_0}\not \cong H_{j_0 -1}$.  Thus the proof is complete.
\end{proof}
Now we present the main result of the paper.
\begin{theorem}\label{ASL}
Let $P$ be a finite $n$-dimensional polyhedron. Suppose that $P$ satisfies the following conditions:
\begin{enumerate}
\item
 $D(\pi_1 (P))=k_1 <\infty$ and every retract of $\pi_1 (P)$ is Hopfian;
\item
$D(H_i (\tilde{P}))=k_i <\infty$  and every retract of $H_i (\tilde{P})$ is Hopfian, for all $2\leq i\leq n$.
\end{enumerate}
Then $D(P)\leq \big( \sum_{i=1}^{n}k_i \big) -n+1$.

Here, $H_i (\tilde{P})$ may be considered as either $\mathbb{Z}$-module or $\mathbb{Z}\pi_1 (P)$-module. If it is considered as $\mathbb{Z}$-module, then the condition ``every retract of $H_i (\tilde{P})$ is Hopfian, for all $2\leq i\leq n$'' holds automatically.
\end{theorem}
\begin{proof}
Put $k= \sum_{i=1}^{n}k_i $. We are going to show that $D(P)\leq k-n+1$. By contrary, assume that there exists the following chain of  length $k$ for $P$
\begin{align}\label{b44}
 X_{k-n+2} <^p X_{k-n+1}<^p \cdots <^p X_2 <^p X_1 \leqslant^d X_0 =P,
\end{align}
with domination maps $d_j :X_{j}\longrightarrow X_{j+1}$ and converse maps $u_j :X_{j+1} \longrightarrow X_{j}$, for all $0\leq j\leq k-1$.

Without loss of generality, each $X_i$  may be assumed to be a CW-complex, not necessarily
finite (by the known results of J.H.C. Whitehead, each space homotopy dominated
by a polyhedron has the homotopy type of some CW-complex,  not necessarily finite (see also \cite{Wall})).

From the chain (\ref{b44}), we  have the following sequence of $r$-images of $\pi_1 (P)$
\begin{align}\label{44}
\pi_1 (X_{k-n+2} )\leqslant^d \pi_1 (X_{k-n+1})\leqslant^d \cdots \leqslant^d \pi_1 (X_2 )\leqslant^d \pi_1 (X_1 )\leqslant^d \pi_1 (P),
\end{align}
and the following sequence of $r$-images of $H_i (\tilde{P})$
\begin{align}\label{45}
 H_i (\tilde{X}_{k-n+2} )\leqslant^d H_i (\tilde{X}_{k-n+1})\leqslant^d \cdots \leqslant^d H_i (\tilde{X}_2 )\leqslant^d H_i (\tilde{X}_1 )\leqslant^d H_i (\tilde{P}),
\end{align}
for all $2\leq i\leq n$.

Suppose that for some $1\leq j_0 \leq k-n+1$, we have
\[
\pi_1 (X_{j_0})\cong \pi_1 (X_{j_0 +1}) \quad \text{and} \quad H_i (\tilde{X}_{j_0})\cong H_i (\tilde{X}_{j_0 +1})
\]
for all $2\leq i\leq n$. Since an epimorphism between two isomorphic Hopfian groups is an isomorphism, one can easily see that the epimorphisms   $\pi_1 (d_{j_0} ) :\pi_1 (X_{j_0})\longrightarrow \pi_1 (X_{j_0 +1})$ and $H_i (d_{j_0} ) :H_i (\tilde{X}_{j_0})\longrightarrow H_i (\tilde{X}_{j_0 +1})$ are isomorphisms.  Now by the Whitehead Theorem, $d_{j_0} :X_{j_0}\longrightarrow X_{j_0 +1}$ is a homotopy equivalence which implies that $X_{j_0}$ and $X_{j_0 +1}$ have the same homotopy type which is a contradiction to $X_{j_0 +1}<^p X_{j_0}$. Accordingly, for each $1\leq j\leq k-n+1$, one can conclude that either $\pi_1 (X_{j})\not \cong \pi_1 (X_{j+1})$ or $H_{i}(\tilde{X}_{j})\not \cong H_i (\tilde{X}_{j+1})$ at least for some $2\leq i\leq n$. Since $D(\pi_1 (P))=k_1$ and $D(H_i (\tilde{P}))=k_i$ $(2\leq i\leq n)$,  by Proposition \ref{2.11}  the chains (\ref{44}) and (\ref{45}) contain at most $k_1$ and $k_i$ non-isomorphic groups, respectively. This shows that there are at most $k_1 -1$ and $k_i -1$ copies of  the sign ``$\not \cong$'' $(2\leq i\leq n)$  in the chains (\ref{44}) and (\ref{45}), respectively. Consequently, there is at most $\sum_{i=1}^{n}(k_i -1 )=k-n$ copies of the sign ``$\not \cong$'' in all chains (\ref{44}) and (\ref{45}) $(2\leq i\leq n)$. Accordingly, by Lemma \ref{ASASI}, it is easy to see that $\pi_1 (X_{k-n+2} )\cong \pi_1 (X_{k-n+1})$ and $H_i (\tilde{X}_{k-n+2} )\cong H_i (\tilde{X}_{k-n+1})$ for all $2\leq i\leq n$. By the Whitehead Theorem, one can conclude that $X_{k-n+2}$ and  $X_{k-n+1}$ have the same homotopy type which is a contradiction to the chain (\ref{b44}). Consequently, we have $D(P)\leq k-n+1$.
\end{proof}
In the next corollary, we present an upper bound for the depths of polyhedra with finite fundamental groups.
\begin{corollary}
Let $P$ be a finite $n$-dimensional polyhedron with finite fundamental group. Then
\[
D(P)\leq D(\pi_1 (P))+ \sum_{i=2}^{n}t_i ,
\]
where $t_i$'s  are  introduced as follows
\[
H_i (\tilde{P},\mathbb{Z})\cong \mathbb{Z}_{p_{i_1}^{\alpha_{i_1}}}\oplus \cdots \oplus \mathbb{Z}_{p_{i_{t_i}}^{\alpha_{i_{t_i}}}} \; \text{for}\; 2\leq i\leq n,
\]
where $p_{i_{j}}^{\alpha_{i_j}}\neq p_{i_{j'}}^{\alpha_{i_{j'}}}$ for $j\neq j'$, $p_{i_j}$'s are prime numbers, $\alpha_{i_j}$'s are non-negative integers, and $\mathbb{Z}_{1}=\mathbb{Z}$.
\end{corollary}
\begin{proof}
Since $\pi_1 (P)<\infty$,  the universal covering space $\tilde{P}$ of a finite polyhedron $P$  is a finite polyhedron. Hence the homology groups $H_i (\tilde{P},\mathbb{Z})$ are  finitely generated abelian groups for all $i$. Now by  part $(ii)$ of Proposition \ref{dfg},  we have $D(H_i (\tilde{P},\mathbb{Z}))=t_i +1$ for all $2\leq i\leq n$. Since all finite groups and finitely generated abelian groups are Hopfian,  we have $D(P)\leq D(\pi_1 (P))+\sum_{i=2}^{n}t_i$ by Theorem \ref{ASL}.
\end{proof}
\begin{corollary}
Le $P$ be an $n$-dimensional polyhedron with abelian fundamental group $\pi_1 (P)\cong  \mathbb{Z}_{p_{1}^{\alpha_{1}}}\oplus \cdots \mathbb{Z}_{p_{t}^{\alpha_{t}}}$, where for $i\neq j$, $p_{i}^{\alpha_i}\neq p_{j}^{\alpha_j}$, $p_i$'s are prime numbers, $\alpha_i$'s are non-negative integers, and $\mathbb{Z}_{1}=\mathbb{Z}$. If $D(H_i (\tilde{P}))=k_i <\infty$ (as $\mathbb{Z}\pi_1 (P)$-module) for all $2\leq i \leq n$, then $D(P)\leq \big( \sum_{i=2}^{n}k_i \big) + t -n +1$.
\end{corollary}
\begin{proof}
Note that $\pi_1 (P)$ is a finitely generated abelian, so $\pi_1 (P)$ is of the form  $\mathbb{Z}_{p_{1}^{\alpha_{1}}}\oplus \cdots \mathbb{Z}_{p_{t_1}^{\alpha_{t_1}}}$, up to isomophism, where for $i\neq j$, $p_{i}^{\alpha_i}\neq p_{j}^{\alpha_j}$, $p_i$'s are prime numbers, $\alpha_i$'s are non-negative integers, and $\mathbb{Z}_{1}=\mathbb{Z}$. We know that $H_i (\tilde{P})$ is a finitely generated $\mathbb{Z}\pi_1 (P)$-module for all $2\leq i \leq n$. But since $\pi_1 (P)$ is abelian, the group ring $\mathbb{Z}\pi_1 (P)$ is a commutative ring. It follows that $H_i (\tilde{P})$ $(2\leq i\leq n)$ is a Hopfian module by \cite[Proposition 1.2]{Vas}. Now by  part $(ii)$ of  Proposition \ref{dfg}, Theorem \ref{ASL} and the hypothesis, the proof is complete.
\end{proof}
\begin{corollary}\label{abelian}
Let $P$ be an $n$-dimensional polyhedron with the abelian fundamental group $\pi_1 (P)\cong  \mathbb{Z}_{p_{1}^{\alpha_{1}}}\oplus \cdots \mathbb{Z}_{p_{t_1}^{\alpha_{t_1}}}$, where for $i\neq j$, $p_{i}^{\alpha_i}\neq p_{j}^{\alpha_j}$, $p_i$'s are prime numbers, $\alpha_i$'s are non-negative integers, and $\mathbb{Z}_{1}=\mathbb{Z}$. Also,suppose that the homology groups
$H_i (\tilde{P};\mathbb{Z})$ are finitely generated of the form
\[
H_i (\tilde{P};\mathbb{Z})\cong  \mathbb{Z}_{p_{i_1}^{\alpha_{i_1}}}\oplus \cdots \oplus  \mathbb{Z}_{p_{i_{t_i}}^{\alpha_{i_{t_i}}}},
\]
where for $j\neq j'$, $p_{i_{j}}^{\alpha_{i_j}}\neq p_{i_{j'}}^{\alpha_{i_{j'}}}$, $p_{i_j}$'s are prime numbers, $\alpha_{i_j}$'s are non-negative integers, and $\mathbb{Z}_{1}=\mathbb{Z}$.  Then the depth of $P$ is at most $\big( \sum_{i=1}^{n}t_i \big) +1$.
\end{corollary}
\begin{proof}
Since $\pi_1 (P)$ and $H_i (\tilde{P},\mathbb{Z})$'s $(2\leq i\leq n)$ are finitely generated abelian groups,  we have $D(\pi_1 (P))=t_1 +1$ and $D(H_i (\tilde{P},\mathbb{Z}))=t_i +1$ by part $(ii)$ of Proposition \ref{dfg}. Now  by Theorem \ref{ASL}, the proof is completed.
\end{proof}
\begin{remark}
Theorem \ref{ASL} in comparison with recent works concerning the depth of polyhedra (see, for example \cite{K6}) has two advantages: one that it presents an upper bound for the depth of some polyhedra in terms of the depths of their fundamental groups and the homology groups of their universal covering spaces, and other one is that it contains wider class of polyhedra. For example, by \cite[Theorem 2]{K6}, every polyhedron with virtually polycyclic fundamental group has finite depth.  But the polyhedron $\mathbb{S}^1 \vee \mathbb{S}^1$ is an example of a polyhedron with non-virtually polycyclic group but finite depth. Note that free groups with finite rank $\geq 2$ are not virtually polycyclic.
\end{remark}

\end{document}